\documentclass{amsart}
\usepackage{mathrsfs}
\usepackage{mathrsfs}
\usepackage{amsfonts}
\usepackage{cases}
\usepackage{latexsym}
\usepackage{amsmath}
\usepackage[all]{xy}
\usepackage{stmaryrd}
\usepackage{amsfonts}
\usepackage{amsmath,amssymb,amscd,bbm,amsthm,mathrsfs,dsfont}

\usepackage{fancyhdr}
\usepackage{amsxtra,ifthen}
\usepackage{verbatim}

\numberwithin{equation}{section}

\theoremstyle{plain}
\newtheorem{theorem}{Theorem}[section]
\newtheorem{lemma}[theorem]{Lemma}
\newtheorem{proposition}[theorem]{Proposition}
\newtheorem{corollary}[theorem]{Corollary}

\theoremstyle{definition}

\begin{document}

\title[Characterizations of Lie Higher Derivations ]
{Characterizations of Lie Higher Derivations on $\mathcal{J}$-Subspace
Lattice Algebras}

\author{Dong Han and Feng Wei}

\address{Han: School of Mathematics and Information Science,
Henan Polytechnic University, Jiaozuo, 454003, P. R. China}

\email{lishe@hpu.edu.cn}

\address{Wei: School of Mathematics and Statistics,
Beijing Institute of Technology, Beijing, 100081, P. R. China}

\email{daoshuo@hotmail.com}\email{daoshuowei@gmail.com}

\begin{abstract}
Let $\mathcal{L}$ be a $\mathcal{J}$-subspace lattice on a Banach space $X$ over the
real or complex field $\mathbb{F}$ and $ \mathrm{Alg}\mathcal{L}$ be the associated $\mathcal{J}$-subspace
lattice algebras. In this paper, we characterize the structure of a family $\{L_n\}_{n=0}^{\infty}: \mathrm{Alg}\mathcal{L}\rightarrow \mathrm{Alg}\mathcal{L}$ of linear mappings satisfying the condition
$$
L_n([A, B])=\sum_{i+j=n}[L_i(A), L_j(B)]
$$ for any $A, B\in\mathrm{Alg}\mathcal{L}$ with $AB = 0$. Moreover,
the family $\{L_n\}_{n=0}^{\infty}: \mathrm{Alg}\mathcal{L}\rightarrow \mathrm{Alg}\mathcal{L}$ of linear mappings satisfying $L_n([A, B]_{\xi})=\sum_{i+j=n}[L_i(A), L_j(B)]_{\xi}
$ for any $A, B\in\mathrm{Alg}\mathcal{L}$ with $AB = 0$ and $1\neq \xi\in \mathbb{F}$ is also considered in the current work.
\end{abstract}



\thanks{This work of the first author is supported by the Doctor Foundation of Henan Polytechnic University (B2010-21), the Natural Science Research Program of Education Department of Henan Province (16A110031 and 15A110026)
and the National Natural Science Foundation of China (No.11301155).}

\maketitle

\section{Introduction}
\label{xxsec1}

Let $\mathcal{A}$ be an associative algebra over a field
$\mathbb{F}$. A linear mapping $\tau:
\mathcal{A}\longrightarrow \mathcal{A}$ is called a
\textit{derivation} if $\tau(AB)=\tau(A)B+A\tau(B)$ for all $A, B\in
\mathcal{A}$. Recall that a linear
mapping $g: \mathcal{A}\rightarrow \mathcal{A}$ is called a
\textit{generalized derivation} if there exists a derivation
$\tau$ such that
$$
g(AB)=g(A)B+A\tau(B)
$$
for all $A, B\in\mathcal{A}$. A linear mapping $L:
\mathcal{A}\longrightarrow \mathcal{A}$ is called a \textit{Lie
derivation} if $L([A, B])=[L(A), B]+[A, L(B)]$ for all $A, B\in
\mathcal{A}$, where $[A, B] = AB-BA$ is the usual Lie product. Clearly, every derivation on $\mathcal{A}$ is a Lie
derivation. But, the converse statement is in general not true. More recently, there have been a number of papers on the study of conditions
under which mappings such as derivations and Lie derivations of noncommutative algebras or operator algebras can be completely determined by the action on some subsets of the given algebras(see
\cite{Bresar, ChebotarKeLeeWong, JiQi, LuJing, LiPanShen, Qi, QiCuiHou, QiHou1, QiHou2, QiHou3, Zhou, ZhuXiongZhang} and the references therein). Of these, the case of Lie-type mappings are very interesting and important.
Let $X$ be a Banach space over the real or complex field $\Bbb{F}$ with ${\rm dim}X\geq3$ and $\mathcal{B}(X)$ be the algebra of all bounded linear operators on $X$. Lu and Jing \cite{LuJing} gave a characterization of Lie derivations on $\mathcal{B}(X)$ by acting on zero products. Let $L: \mathcal{B}(X)\rightarrow \mathcal{B}(X)$ be a linear mapping satisfying $L([A,B]) = [L(A),B]+[A,L(B)]$ for any $A, B \in \mathcal{B}(X)$ with $AB = 0$ (resp. $AB = P$, here $P$ is a fixed nontrivial idempotent). Then $L=d+\tau$ ,where $d$ is a derivation of $\mathcal{B}(X)$ and $\tau : \mathcal{B}(X)\rightarrow \Bbb{F}I$ is a linear mapping vanishing
at commutators $[A,B]$ with $AB = 0$ (resp. $AB = P$). The related problem has also been investigated for prime rings and triangular algebras in \cite{JiQi, QiHou3}, for $\mathcal{J}$-subspace
lattice algebras in \cite{Qi, QiHou1}, respectively.
However, people pay much less attention to the structure of Lie-type
higher derivations on algebras by local action. The
objective of this article is to describe the structure of Lie($\xi$-Lie)
higher derivations on $\mathcal{J}$-subspace
lattice algebras by acting on zero products.

Let us first recall some basic facts related to Lie higher
derivations of an associative algebra $\mathcal{A}$. Let $\mathbb{N}$ be
the set of all non-negative integers and
$G=\{L_k\}_{k=0}^\infty$  be a family of linear
mappings of $\mathcal{A}$ such that $L_0=id_{\mathcal{A}}$. $G$ is
called:
\begin{enumerate}
\item[(i)] a \emph{higher derivation} if
$$
L_k(xy)=\sum_{i+j=k}L_i(x)L_j(y)
$$
for all $x, y\in\mathcal{A}$ and for each $k\in\mathbb{N}$;

\item[(ii)] a \emph{Lie higher derivation} if
$$
L_k([x, y])=\sum_{i+j=k}[L_i(x), L_j(y)]
$$
for all $x, y\in\mathcal{A}$ and for each $k\in\mathbb{N}$;

\item [(iii)] a \emph{generalized higher derivation} if there exists a higher derivation $D=\{d_i\}_{i\in \mathbb{N}}$ such that
$$
L_k(xy)=\sum_{i+j=k}L_i(x)d_j(y)
$$
for all $x,y\in \mathcal {A}$ and for each $k\in \mathbb{N}$. Then $D$ is called an \textit{associated higher derivation} of $G$.
\end{enumerate}
Note that $L_1$ is always a Lie derivation if $D=\{L_k\}_{k\in\mathbb{N}}$ is a Lie higher
derivation. Obviously, every
higher derivation is a Lie higher derivation. But the converse
statements are in general not true. Various higher derivations,
which consist of a family of some additive mappings, frequently appear
in commutative and noncommutative contexts( see
\cite{FerreroHaetinger1, FerreroHaetinger2, Han, Han1, Hazewinkel, LiGuo, LiPanShen, QiHou4, WeiXiao, XiaoWei1, XiaoWei2} and so on).
In \cite{Han1} the first author gave a characterization concerning Lie-type higher derivations of associative algebras. These properties
easily enable us to transfer the problems of Lie-type higher derivations into the same problems related
to Lie-type derivations. In this method, the first author describe Lie-type higher derivations on different operator algebras easily. However, it seems difficult to give a unified approach in characterizing Lie-type higher derivations on associative algebras by local action.
So in terms of the problem studied in the current work, namely, describing the form of  the Lie-type
higher mappings on operator algebras by local action, the method mentioned above is not applicable. Hence we pay our attention to Lie($\xi$-Lie) higher derivations on $\mathcal{J}$-subspace
lattice algebras by acting on zero products in this article.

Let $X$ be a Banach space over the real or complex field
$\mathbb{F}$. A family $\mathcal{L}$ of subspaces of $X$ is called a
subspace lattice of $X$ if it contains $\{0\}$ and $X$, and is
closed under the operations closed linear span $\bigvee$ and
intersection $\bigwedge$ in the sense that $\bigvee_{\gamma\in
\Gamma}L_{\gamma}\in \mathcal{L}$ and $\bigwedge_{\gamma\in
\Gamma}L_{\gamma}\in \mathcal{L}$ for every family
$\{L_{\gamma}\colon \gamma\in \Gamma\}$ of elements in
$\mathcal{L}$. For a subspace lattice $\mathcal{L}$ of $X$, the
associated subspace lattice algebra $\mathrm{Alg}\mathcal{L}$ is the
set of operators on $X$ leaving each subspace in
$\mathcal{L}$ invariant. As the above definition, for an arbitrary
subspace $K$ in $\mathcal{L}$ we can establish
$$
K_- =\bigvee \{ L \in \mathcal{L}: K \nsubseteq L \}.
$$
The class of $\mathcal{J}$-subspace lattices was defined by Panaia
in his dissertation \cite{Panaia} and covers atomic Boolean
subspace lattices and pentagon subspace lattices.
$\mathcal{J}$-subspace lattices are a particular sort of
complemented lattice, satisfying certain other criteria. To be
precise, define
$$
\mathcal{J(L)}= \{K \in {\mathcal L}\colon K \neq \{0\}\ \text{and}\
K_-\neq X \}.
$$
Then $\mathcal{L}$ is called a \textit{$\mathcal{J}$-subspace
lattice}(simply, JSL) on $X$, provided all of the following conditions are
satisfied:

(1)$\bigvee \{K\colon K \in \mathcal{J(L)}\}=X$;

(2)$\bigwedge\{K_- \colon K \in \mathcal{J(L)} \}=\{0\}$;

(3)$K \bigvee K_-=X$ for each $K$ in $\mathcal{J(L)}$;

(4)$K\bigwedge K_-=\{0\}$ for each $K$ in $\mathcal{J(L)}$.

\noindent If $\mathcal{L}$ is a $\mathcal{J}$-subspace lattice, the associated
subspace lattice algebra ${\rm Alg}\mathcal{L}$ is called a
\textit{$\mathcal{J}$-subspace lattice algebra}(JSL algebras). The center is defined usually by $Z({\rm Alg}\mathcal{L}) = \{Z \in {\rm Alg}\mathcal{L}: ZA = AZ \ \text{for all} \ A \in {\rm Alg}\mathcal{L}\}$.

The outline of our paper is organized as follows. Let $\mathcal{L}$ be a $\mathcal{J}$-subspace lattice on a Banach space $X$ over the real or complex field $\Bbb{F}$ and $\mathrm{Alg}\mathcal{L}$ be the associated $\mathcal{J}$-subspace lattice algebra. In the
second section we give a characterization of a family $\{L_n\}_{n=0}^\infty$ of linear mappings satisfying
$$
L_n([A, B])=\sum_{i+j=n}[L_i(A), L_j(B)]
$$
for any $A, B \in  {\rm Alg}\mathcal{L}$ with $AB=0$. This permits us to transfer the
problems related to Lie higher derivations by acting on zero products into the same problems of the corresponding Lie derivations. Then this result is applied to describe the Lie higher derivations by acting on zero products on ${\rm Alg}\mathcal{L}$. Basing on the second section, we continue to investigate $\xi$-Lie higher derivations by acting on zero products in the third section, where $1\neq \xi\in \Bbb{F}$.

\section{Characterizations of Lie higher derivations by acting on zero products}
\label{xxsec2}

For any $A \in \mathcal{B}(X)$,
denote by $A^*$ the adjoint of $A$. For $x \in X$ and $f \in X^*$, the rank-one operator
$x \otimes f$ is defined by $(x \otimes f)y=f(y)x$ \ \text{for all}\   $y \in X$. For any non-empty subset $\mathfrak{L} \in X$,
$\mathfrak{L}^{\bot}$
denotes its annihilator, that is, $\mathfrak{L}^{\bot} = \{f \in X^* : f(x) = 0 \ \text{for all} \ x \in \mathfrak{L}\}$.
Now Let us first give some lemmas related to rank-one operator, which is crucial to what follows.

\begin{lemma} {\rm (}\cite{Longstaff}{\rm )}\label{xxsec2.l1}. Let $\mathcal{L}$ be a $\mathcal{J}$-subspace lattice on a Banach space $X$. Then
$x\otimes f\in \mathrm{Alg}\mathcal{L}$ if and only if there exists a subspace $K\in \mathcal{J}(\mathcal{L})$ such that $x \in K$
and $f \in K_-^{\bot} $.
\end{lemma}

\begin{lemma} {\rm (}\cite{LongstaffPanaia}{\rm )}\label{xxsec2.l2}. Let $\mathcal{L}$ be a $\mathcal{J}$-subspace lattice on a Banach space $X$ and
$K\in \mathcal{J}(\mathcal{L})$. Then, for any nonzero vector $x \in K$
, there exists $f \in K_-^{\bot} $
such
that $f(x) =1$; dually, for any nonzero functional  $f \in K_-^{\bot} $
, there exists $x \in K$
such that $f(x) =1$.
\end{lemma}

\begin{lemma} {\rm (}\cite{Qi}{\rm )}\label{xxsec2.l3}. Every rank one operator  $x\otimes f\in \mathrm{Alg}\mathcal{L}$ is a linear combination
of idempotents in $\mathrm{Alg}\mathcal{L}$.
\end{lemma}

 Moreover, finite-rank operators will also be used in later. Given a subspace lattice  $\mathcal{L}$, by $\mathcal{F}_{\mathcal{L}}(K)$ we denote the subspace spanned by all rank one operators $x \otimes f$ with
$x \in K$ and $f \in K_-^{\bot}$ for arbitrary $K\in \mathcal{J}(\mathcal{L})$.
$\mathcal{F}(\mathcal{L})$ denotes the algebra of all finite rank operators in
$\mathrm{Alg}\mathcal{L}$.

\begin{lemma}{\rm (}\cite{LuLi} or \cite{Panaia}{\rm )}\label{xxsec2.l4}. Let $\mathcal{L}$ be a $\mathcal{J}$-subspace lattice on a Banach space $X$. Suppose that $A$ is an operator of rank $n$
in $\mathcal{F}(\mathcal{L})$. Then $A$ can be written as a sum of n rank-1 operators in $ \mathrm{Alg}\mathcal{L}$.
\end{lemma}

Let $\{L_n\}_{n=0}^{\infty}: \mathrm{Alg}\mathcal{L}\rightarrow \mathrm{Alg}\mathcal{L}$ be a family of linear mappings such that
$$
L_n([A, B])=\sum_{i+j=n}[L_i(A), L_j(B)]
$$ for any $A, B\in\mathrm{Alg}\mathcal{L}$ with $AB = 0$. The following result shows that the restriction of $\{L_n\}_{n=0}^{\infty}$  to $\mathcal{F}(\mathcal{L})$ is actually a Lie higher derivation.

\begin{lemma}\label{xxsec2.l5}
Let $\mathcal{L}$ be a $\mathcal{J}$-subspace lattice on a Banach space $X$ over the
real or complex field $\mathbb{F}$ and $ \mathrm{Alg}\mathcal{L}$ be the associated $\mathcal{J}$-subspace
lattice algebra. Suppose that $\{L_n\}_{n=0}^{\infty}: \mathrm{Alg}\mathcal{L}\rightarrow \mathrm{Alg}\mathcal{L}$ is a family of linear mappings such that
$$
L_n([A, B])=\sum_{i+j=n}[L_i(A), L_j(B)]
$$ for any $A, B\in\mathrm{Alg}\mathcal{L}$ with $AB = 0$. Then  $$
L_n([A,F])=[L_n(A),F]+\sum_{\substack{i+j=n \\ 0<i,j<n}}[L_i(A), L_j(F)]+[A, L_n(F)]
$$
for all $A \in \mathrm{Alg}\mathcal{L}$ and  $F\in \mathcal{F}(\mathcal{L})$.
\end{lemma}

\begin{proof} The lemma will be proved through two claims.

\textbf{Claim 1.} $L_{n}(\mathbb{F}I) \subseteq Z(\mathrm{Alg}\mathcal{L})$.

Let us show this claim by induction on the index $n$. When $n=1$, $L_1$ satisfies
$L_1([A,B]) = [L_1(A),B] + [A,L_1(B)]$ for any $A,B \in \mathrm{Alg}\mathcal{L}$ with $AB = 0$, and $L_1(\mathbb{F}I) \subseteq Z(\mathrm{Alg}\mathcal{L})$ (see \cite{Qi}).
Now let $s\in  \mathbb{N}$ with $s\geq 1$ and we assume that the conclusion holds for all $s\leq n$.

For any scalar $\lambda$ and any idempotent $P \in \mathrm{Alg}\mathcal{L}$, in view of the fact $\lambda P(I - P) = 0$, we have
$$
\begin{aligned}
 L_{n+1} ([\lambda P,I-P])&=[L_{n+1} (\lambda P),I-P]+[\lambda P,L_{n+1} (I-P)]
 \\ \ \ &+\sum_{\substack{i+j={n+1}  \\ 0<i,j<{n+1} }}[L_i(\lambda P), L_j(I-P)]\\
 &=PL_{n+1} (\lambda P) - L_{n+1} ( \lambda P)P + \lambda PL_{n+1} (I) - \lambda PL_{n+1} (P)\\
 &\ \ - \lambda L_{n+1} (I)P+ \lambda L_{n+1} (P)P+ \sum_{\substack{i+j={n+1}  \\ 0<i,j<{n+1} }}[L_i(\lambda P), L_j(I-P)].
\end{aligned}
$$
On the other hand, by $(\lambda I - \lambda P)P = 0$ one can assert
$$
\begin{aligned}
 L_{n+1} ([\lambda I - \lambda P,P])&=[L_{n+1} (\lambda I - \lambda P),P]+[\lambda I - \lambda P,L_{n+1} (P)]\\
 & \ \ +\sum_{\substack{i+j={n+1}  \\ 0<i,j<{n+1} }}[L_i(\lambda I - \lambda P), L_j(P)]\\
 &=L_{n+1} (\lambda I)P - L_{n+1} (\lambda P)P - PL_{n+1} (\lambda I) + PL_{n+1} (\lambda P)\\
 &- \lambda PL_{n+1} (P) + \lambda L_{n+1} (P)P+\sum_{\substack{i+j={n+1}  \\ 0<i,j<{n+1} }}[L_i(\lambda I - \lambda P), L_j(P)].
\end{aligned}
$$
Comparing the above two equations, we get from the induction hypothesis that
\begin{equation}\label{xxsec2e1}
\lambda PL_{n+1}(I) - \lambda L_{n+1}(I)P = L_{n+1}(\lambda I)P - PL_{n+1}(\lambda I).
\end{equation}
A similar discussion as in \cite{Qi} shows that $L_{n+1}(\mathbb{F}I) \subseteq Z(\mathrm{Alg}\mathcal{L})$.

\textbf{Claim 2.}
For any $A \in \mathrm{Alg}\mathcal{L}$ and any rank one operator $x\otimes f\in \mathrm{Alg}\mathcal{L}$, we have
$$
L_n([A,x\otimes f])=[L_n(A),x\otimes f]+\sum_{\substack{i+j=n \\ 0<i,j<n}}[L_i(A), L_j(x\otimes f)]+[A, L_n(x\otimes f)].
$$

Take any $A \in \mathrm{Alg}\mathcal{L}$ and any idempotent $P \in \mathrm{Alg}\mathcal{L}$. For any scalar $\lambda$,
notice that $AP(\lambda I - \lambda P) = 0$. By Claim 1 it follows that
$$
\begin{aligned}
& L_n(\lambda PAP)-L_n(\lambda AP)=L_n([AP, \lambda I-\lambda P])\\
 &=[L_n(AP),\lambda I-\lambda P]+\sum_{\substack{i+j=n \\ 0<i,j<n}}[L_i(AP), L_j(\lambda I-\lambda P)]+[AP, L_n(\lambda I-\lambda P)]\\
 &=\lambda PL_n(AP) - \lambda L_n(AP)P - APL_n( \lambda P) + L_n(\lambda P)AP\\
 &+\sum_{\substack{i+j=n \\ 0<i,j<n}}[L_i(AP), L_j(-\lambda P)]
\end{aligned}
$$
However, since $(A -AP)\lambda P = 0$, we obtain
$$
\begin{aligned}
& L_n(\lambda PAP)-L_n(\lambda PA)=L_n([A-AP, \lambda P])\\
 &=[L_n(A-AP), \lambda P]+\sum_{\substack{i+j=n \\ 0<i,j<n}}[L_i(A-AP), L_j( \lambda P)]+[A-AP, L_n( \lambda P)]\\
 &=\lambda L_n(A)P -  \lambda PL_n(A) + \lambda PL_n(AP) - \lambda L_n(AP)P\\
 &+AL_n(\lambda P) - APL_n(\lambda P) - L_n(\lambda P)A + L_n(\lambda P)AP\\
 &+\sum_{\substack{i+j=n \\ 0<i,j<n}}[L_i(A-AP), L_j(\lambda P)]
\end{aligned}
$$
 Comparing the last two relations leads to
\begin{equation}\label{xxsec2e2}
L_n([A,\lambda P])=[L_n(A),\lambda P]+\sum_{\substack{i+j=n \\ 0<i,j<n}}[L_i(A), L_j(\lambda P)]+[A,  L_n(\lambda P)].
\end{equation}
for all $A \in \mathrm{Alg}\mathcal{L}$. Now, by Lemma \ref{xxsec2.l3},
the claim is true. Furthermore, taking into account Lemma \ref{xxsec2.l4}, we can further get
$$
L_n([A,F])=[L_n(A),F]+\sum_{\substack{i+j=n \\ 0<i,j<n}}[L_i(A), L_j(F)]+[A, L_n(F)]
$$
for all $A \in \mathrm{Alg}\mathcal{L}$ and  $F\in \mathcal{F}(\mathcal{L})$.
\end{proof}

We must indicate that the proofs in the following lemma is essentially the same as those in \cite{Lu}, but it is done in a slightly different way.

\begin{lemma}\label{xxsec2.l7}
Let $\mathcal{L}$ be a $\mathcal{J}$-subspace lattice on a Banach space $X$ over the real or complex field $\Bbb{F}$ and $\mathrm{Alg}\mathcal{L}$ be the associated $\mathcal{J}$-subspace lattice algebra. Suppose that $\delta: \mathrm{Alg}\mathcal{L}\rightarrow \mathrm{Alg}\mathcal{L}$ is a family of linear mappings such that
 $\delta([A, F])=[\delta(A), F]+[A, \delta(F)]$ for all $A\in \mathrm{Alg}\mathcal{L}$, $F\in \mathcal{F}(\mathcal{L})$ and $n\in \mathbb{N}$.
Then for any $K \in \mathcal{J}(\mathcal{L})$, there is an operator
$S$ in $\mathcal{F}_{\mathcal{L}}(K)$ and an operator $\tau(P)$ in $Z( \mathrm{Alg}\mathcal{L})$ such that $ \delta(x\otimes f) = [x\otimes f, S]+\tau(x\otimes f)$ for all $x \in K$ and $f\in K_-^{\bot}$.
\end{lemma}

\begin{proof}
Let $P$ be an idempotent operator in $\mathcal{F}_{\mathcal{L}}(K)$. Set $P_1 =P$ and $P_2 = I - P$.
Then for any $A_{11} \in P_1\mathrm{Alg}\mathcal{L}P_1$ we have
$$0 = \delta([P_1,A_{11}]) = \delta(P_1)A_{11} - A_{11} \delta(P_1) + P_1\delta(A_{11})- \delta(A_{11})P_1.$$
Multiplying the above equality by $P_1$ from left and right sides gives
$$
P_1\delta(P_1)P_1A_{11} = A_{11}P_1\delta(P_1)P_1.
$$
In an analogous manner, we also get
$$
\begin{aligned}
&P_2\delta(P_1)P_2A_{22} = A_{22}P_2\delta(P_1)P_2,\\
 &P_1\delta(P_1)P_1A_{12} = A_{12}P_2\delta(P1)P_2,\\
 &P_2\delta(P_1)P_2A_{21} = A_{21}P_1\delta(P_1)P_1.
\end{aligned}
$$
These facts imply that $ \tau (P) := P_1\delta(P_1)P_1 + P_2\delta(P_1)P_2 \in Z(\mathrm{Alg}\mathcal{L})$. Now let $S=P_1\delta(P_1)P_2 - P_2\delta(P_1)P_1$, then it is easy to check $S=[\delta(P_1),P_2]=[P_1 ,\delta(P_1)]$. Note that $\mathcal{F}_{\mathcal{L}}(K)$ is an ideal of $\mathrm{Alg}\mathcal{L}$(see Lemma 2.4 in \cite{Lu}), so
it is easy to check that $S$ in $\mathcal{F}_{\mathcal{L}}(K)$ and $\delta(P) = [P, S] + \tau (P)$.

\textbf{Case 1.} If $f(x)\neq 0$, then the result follows from the linearity.

\textbf{Case 2.} Now assume that $f(x) = 0$. In light of Lemma \ref{xxsec2.l2} one can take $y \in K$ such that $f(y) = 1$.
By the previous fact it follows that $\delta(y\otimes f) = [y\otimes f, [y\otimes f,\delta(y\otimes f)]] + \tau(y\otimes f)$.
Let us define a mapping $\Delta: \mathrm{Alg}\mathcal{L}\rightarrow \mathrm{Alg}\mathcal{L}$  by
\begin{equation}\label{xxsec2e8}\Delta(A) =\delta(A) - [A,[y\otimes f,\delta(y\otimes f)]],\end{equation}
where $A\in \mathrm{Alg}\mathcal{L}$. Then $\Delta(y\otimes f)\in Z(\mathrm{Alg}\mathcal{L})$ and therefore,
$$
\begin{aligned}
&\ \Delta(x\otimes f) =\Delta([x\otimes f,y\otimes f]) =[\Delta(x\otimes f),y\otimes f]\\
 &=\Delta(x\otimes f)y\otimes f-y\otimes f\Delta(x\otimes f),\\
\end{aligned}
$$
which implies that $\Delta(x\otimes f)=(I-y\otimes f)\Delta(x\otimes f)y\otimes f$. Furthermore, $ \Delta(x\otimes f)$ can be rewritten as $z\otimes f$, where
$z\in K$ and $f(z)=0$.  Choosing $g \in K_-^{\bot}$ such that $g(x)=1$,
we arrive at
\begin{equation}\label{xxsec2e9}
\Delta(x\otimes f)=z\otimes f=(z\otimes g)(x\otimes f)-(x\otimes f)(z\otimes g).
\end{equation}
Combining (\ref{xxsec2e8}) with (\ref{xxsec2e9}) yields
$$
\begin{aligned}
\delta(x\otimes f)&=[x\otimes f,[y\otimes f,\delta(y\otimes f)]]-[x\otimes f,z\otimes g]\\
 &=[x\otimes f,[y\otimes f,\delta(y\otimes f)]-z\otimes g],
\end{aligned}
$$
from which we can see that the conclusion still remains to be established in this case.
\end{proof}

The following proposition will give a new characterization to the family of linear mappings $\{L_n\}_{n=0}^{\infty}: \mathrm{Alg}\mathcal{L}\rightarrow \mathrm{Alg}\mathcal{L}$ satisfying
$$
L_n([A, B])=\sum_{i+j=n}[L_i(A), L_j(B)]
$$
for any $A, B\in\mathrm{Alg}\mathcal{L}$ with $AB = 0$. These properties will
partly enable us to transfer the problems of $\{L_n\}_{n=0}^{\infty}$ into the same problems related
to Lie derivations on $\mathcal{J}$-subspace
lattice algebras by acting on zero product.

\begin{proposition}\label{xxsec2.p1}
Let $\mathcal{L}$ be a $\mathcal{J}$-subspace lattice on a Banach space $X$ over the
real or complex field $\mathbb{F}$ and $\mathrm{Alg}\mathcal{L}$ be the associated $\mathcal{J}$-subspace
lattice algebra. Suppose that $\{L_n\}_{n=0}^{\infty}: \mathrm{Alg}\mathcal{L}\rightarrow \mathrm{Alg}\mathcal{L}$ is a family of linear mappings satisfying
$$
L_n([A, B])=\sum_{i+j=n}[L_i(A), L_j(B)]
$$ for any $A, B\in\mathrm{Alg}\mathcal{L}$ with $AB = 0$.  Then there is a sequence $\{\delta_n\}_{n=0}^\infty :\mathrm{Alg}\mathcal{L}\rightarrow \mathrm{Alg}\mathcal{L}$ of linear mappings satisfying \begin{equation}
\label{xxsec2e9c1}\delta_n([A, F])=[\delta_n(A), F]+[A, \delta_n(F)]\end{equation}
for all $A\in \mathrm{Alg}\mathcal{L}$ and $F\in \mathcal{F}(\mathcal{L})$
such that
$$
(n+1)L_{n+1}=\sum_{k=0}^nL_{n-k}  \delta_{k+1}
$$
for each non-negative integer $n$.
\end{proposition}

\begin{proof}
Let us prove this proposition by induction on the index $n$. If $n=0$, then
$$
L_1([A, B])=[L_1(A), L_0(B)]+[L_0(A),L_1(B)]=[L_1(A), B]+[A,
L_1(B)]
$$
for any $A, B \in \mathrm{Alg}\mathcal{L}$ with $AB= 0$. If we set $\delta_1=L_1$, then
$\delta_1([A, B])=[\delta_1(A), B]+[A, \delta_1(B)]$ for any $A, B \in \mathrm{Alg}\mathcal{L}$ with $AB= 0$. By Lemma \ref{xxsec2.l5} we know that  $\delta_1([A, F])=[\delta_1(A), F]+[A, \delta_1(F)]$ holds for all $A\in \mathrm{Alg}\mathcal{L}$ and $F\in \mathcal{F}(\mathcal{L})$.

We now suppose that $\delta_k$ is a well-established mapping of $ \mathrm{Alg}\mathcal{L}$ for each $k\leq n$. Let us
define
$$
\delta_{n+1}=(n+1)L_{n+1}-\sum_{k=0}^{n-1}L_{n-k}\delta_{k+1}.
$$

It is sufficient to show that $\delta_{n+1}: \mathrm{Alg}\mathcal{L}\rightarrow \mathrm{Alg}\mathcal{L}$ satisfies
the equality (\ref{xxsec2e9c1}).
For any  $A \in \mathrm{Alg}\mathcal{L}$ and any rank one operator $x\otimes f\in \mathrm{Alg}\mathcal{L}$,
by the induction hypothesis, we can compute that
$$
\begin{aligned}
\delta_{n+1}([A,x\otimes f])&=(n+1)L_{n+1}([A,x\otimes f])-\sum_{k=0}^{n-1}L_{n-k}\delta_{k+1}([A,x\otimes f])\\
&=(n+1)L_{n+1}([A,x\otimes f])-\sum_{k=0}^{n-1}L_{n-k}[\delta_{k+1}(A),x\otimes f]\\&\ \ -\sum_{k=0}^{n-1}L_{n-k}[A,\delta_{k+1}(x\otimes f)].
\end{aligned}
$$
Applying Lemma \ref{xxsec2.l7}, we can further get
$$
\begin{aligned}
\delta_{n+1}([A,x\otimes f])&=(n+1)L_{n+1}([A,x\otimes f])-\sum_{k=0}^{n-1}L_{n-k}[\delta_{k+1}(A),x\otimes f]\\&
\ \ -\sum_{k=0}^{n-1}L_{n-k}[A,(x\otimes f){S}_{k+1}-{S}_{k+1}(x\otimes f)]\\
&=(n+1)\sum_{k=0}^{n+1}[L_k(A),
L_{n+1-k}(x\otimes f)]-\sum_{k=0}^{n-1}\sum_{i=0}^{n-k}[L_i\delta_{k+1}(A),
L_{n-k-i}(x\otimes f)]\\
&\ \ -\sum_{k=0}^{n-1}\sum_{i=0}^{n-k}[L_i(A),
L_{n-k-i}((x\otimes f){S}_{k+1}-{S}_{k+1}(x\otimes f))].
\end{aligned}
$$
For convenience, let us write
$$
\begin{aligned}
U&=\sum_{k=0}^{n+1}k[L_k(A), L_{n+1-k}(x\otimes f)]-\sum_{k=0}^{n-1}\sum_{i=0}^{n-k}[L_i\delta_{k+1}(A), L_{n-k-i}(x\otimes f)],\\
V&=\sum_{k=0}^{n+1}(n+1-k)[L_k(A),
L_{n+1-k}(x\otimes f)]-\sum_{k=0}^{n-1}\sum_{i=0}^{n-k}[L_i(A),
L_{n-k-i}((x\otimes f){S}_{k+1}-{S}_{k+1}(x\otimes f))].
\end{aligned}
$$
Then it is easy to verify $\delta_{n+1}([A,x\otimes f])=U+V$. In the expression of sum
$\sum_{k=0}^{n-1}\sum_{i=0}^{n-k}$, we know that $k\neq n$ and
$0\leq k+i\leq n$. If we set $r=k+i$, then
$$\begin{aligned}
U&=\sum_{k=0}^{n+1}k[L_k(A), L_{n+1-k}(x\otimes f)]-\sum_{r=0}^{n}\sum_{0\leq k\leq r,k\neq n}[L_{r-k}\delta_{k+1}(A),L_{n-r}(x\otimes f)]\\
&=\sum_{r=0}^{n}(r+1)[L_{r+1}(A),
L_{n-r}(x\otimes f)]-\sum_{k=0}^{n-1}[L_{n-k}\delta_{k+1}(A), x\otimes f]\\
&\hspace{10pt}- \sum_{r=0}^{n-1}\sum_{k=0}^{r}[L_{r-k}\delta_{k+1}(A),
L_{n-r}(x\otimes f)]\\
&=\sum_{r=0}^{n-1}[(r+1)L_{r+1}(A)-\sum_{k=0}^{r}L_{r-k}\delta_{k+1}(A),L_{n-r}(x\otimes f)]\hspace{45pt}\\
&\quad +(n+1)[L_{n+1}(A),
x\otimes f]-\sum_{k=0}^{n-1}[L_{n-k}\delta_{k+1}(A), x\otimes f].
\end{aligned}$$
Applying  the induction hypothesis to the above equality,  we obtain
$$
U=[(n+1)L_{n+1}(A)-\sum_{k=0}^{n-1}L_{n-k}\delta_{k+1}(A),
x\otimes f]=[\delta_{n+1}(A), x\otimes f].
$$
On the other hand, a direct computation gives
$$\begin{aligned}
V&=\sum_{i=0}^{n}[L_i(A), (n+1-i)L_{n+1-i}(x\otimes f)] -\sum_{i=1}^{n}\sum_{k=0}^{n-i}[L_i(A), L_{n-k-i}((x\otimes f){S}_{k+1}-{S}_{k+1}(x\otimes f))]\\
&\ \ \ \ -\sum_{k=0}^{n-1}[A, L_{n-k}((x\otimes f){S}_{k+1}-{S}_{k+1}(x\otimes f))]\\
&=\sum_{i=1}^{n}[L_i(A), (n+1-i)L_{n+1-i}(x\otimes f) -\sum_{k=0}^{n-i}L_{n-k-i}((x\otimes f){S}_{k+1}-{S}_{k+1}(x\otimes f))]\\
&\ \ \ +[A, (n+1)L_{n+1}(x\otimes f)]-\sum_{k=0}^{n-1}[A, L_{n-k}((x\otimes f){S}_{k+1}-{S}_{k+1}(x\otimes f))]
\end{aligned}$$
Furthermore, using the induction hypothesis again, we obtain
\begin{equation}\label{xxsec2e9c5}
\begin{aligned}
V&=\sum_{i=1}^{n}[L_i(A),
(n+1-i)L_{n+1-i}(x\otimes f)-\sum_{k=0}^{n-i}L_{n-k-i} \delta_{k+1}(x\otimes f)]\\
&\ \ +[A,(n+1)L_{n+1}(x\otimes f)-\sum_{k=0}^{n-1}L_{n-k} \delta_{k+1}(x\otimes f)]]\\
&\ \ +\sum_{i=1}^{n}[L_i(A),\sum_{k=0}^{n-i}L_{n-k-i} \tau_{k+1}(x\otimes f)]+[A,\sum_{k=0}^{n-1}L_{n-k} \tau_{k+1}(x\otimes f)]\\
&=[A, (n+1)L_{n+1}(x\otimes f)-\sum_{k=0}^{n-1}L_{n-k} \delta_{k+1}(x\otimes f)]\\
&\ \ +\sum_{i=1}^{n}[L_i(A),\sum_{k=0}^{n-i}L_{n-k-i} \tau_{k+1}(x\otimes f)]+[A,\sum_{k=0}^{n-1}L_{n-k} \tau_{k+1}(x\otimes f)].
\end{aligned}
\end{equation}
Note that
$$
\begin{aligned}
&\sum_{i=1}^{n}[L_i(A),\sum_{k=0}^{n-i}L_{n-k-i} \tau_{k+1}(x\otimes f)]+[A,\sum_{k=0}^{n-1}L_{n-k} \tau_{k+1}(x\otimes f)]\\
&=[A,L_{n} \tau_{1}(x\otimes f)+L_{n-1} \tau_{2}(x\otimes f)+\cdots+L_{1} \tau_{n}(x\otimes f)]\\
& \ \ +[L_1(A),L_{n-1} \tau_{1}(x\otimes f)+L_{n-2} \tau_{2}(x\otimes f) +\cdots+L_{1} \tau_{n-1}(x\otimes f)+L_{0} \tau_{n}(x\otimes f)]\\
& \ \ +\cdots\\
& \ \ +[L_{n-1}(A),L_{1} \tau_{1}(x\otimes f) + \tau_{2}(x\otimes f)]\\
& \ \ +[L_n(A), \tau_{1}(x\otimes f)]\\
&=L_n[A,\tau_1(x\otimes f) +L_{n-1}[A,\tau_2(x\otimes f)]+\cdots+L_1[A,\tau_n(x\otimes f)],
\end{aligned}
$$
Taking into account the relation (\ref{xxsec2e9c5}) yields
$$
\begin{aligned}
V&=[A, (n+1)L_{n+1}(x\otimes f)-\sum_{k=0}^{n-1}L_{n-k}(\delta_{k+1}(x\otimes f)]\\
&\ \ +L_n[A,\tau_1(x\otimes f)]+L_{n-1}[A,\tau_2(x\otimes f)]+\cdots+L_1[A,\tau_n(x\otimes f)]\\
 &=[A, \delta_{n+1}(x\otimes f)].
\end{aligned}
$$
Finally, we conclude that
$$
\delta_{n+1}([A, x\otimes f])=U+V=[\delta_{n+1}(A), x\otimes f]+[A, \delta_{n+1}(x\otimes f)].
$$
It follows from Lemma \ref{xxsec2.l4}  that $\delta_{n+1}$ satisfies (\ref{xxsec2e9c1}) due to the linearity.
\end{proof}

 Before giving our main result we recall some of the basic concepts about inner higher derivations which will be used in latter. Let $\mathcal{A}$ be an associative algebra. We denote the set of higher derivations of order $m$ on  $\mathcal{A}$ by $D_m(\mathcal{A})$, which is a group under the multiplication $*$ defined by
$$
(d*d')_n=\sum_{\substack{i+j=n}}d_i\circ d'_j,\ \ n\leq m,
$$
where $d, d'\in D_m(\mathcal{A})$(see \cite{Nowicki} and the
references therein). Let $\mathbf{a}=(a_n)_{n\leq m}$ be a sequence in $\mathcal{A}$.  Denote by $\Delta(\mathbf{a})$ a family of mappings defined by
$$\Delta(\mathbf{a})_n=([a_1,1]*[a_2, 2]*\cdots*[a_n, n])_n, \ \  n\leq m \ \ \ \eqno(\bigstar)$$
where $$[a,k]_n(x)=\left\{
                    \begin{array}{ll}
                       x, & \hbox{if $n=0$ ;} \\
                       0, & \hbox{if $k\nmid n$;} \\
                      a^rx-a^{r-1}xa , & \hbox{if $n\neq0$, and $ n=kr$.}
                    \end{array}
                  \right.
$$
Then $\Delta(\mathbf{a})\in D_m(\mathcal{A})$ is called an \textit{inner higher derivation} of order $m$ (\cite{Nowicki}).

In the context of this article, $\mathbf{T}_K=(T_{Kn})_{n\in \mathbb{N}}$ is a sequence in $\mathcal{B}(K)$ and $\Delta(\mathbf{T})$ is a family of mappings defined by $$\Delta(\mathbf{T})_{Kn}=([{T_K}_1,1]*[{T_K}_2, 2]*\cdots*[{T_K}_n, n])_{n}. $$ For example, let $A  \in \mathrm{Alg}\mathcal{L}$, then
$$\begin{aligned}
&\Delta(\mathbf{T})_{K1}(A)=T_{K1}A-AT_{K1}\\
&\Delta(\mathbf{T})_{K2}(A)=T_{K1}^2A-T_{K1}AT_{K1}+T_{K2}A-AT_{K2},\\
&\Delta(\mathbf{T})_{K3}(A)=T_{K1}^3A-T_{K1}^2AT_{K1}+T_{K1}T_{K2}A+A T_{K2}T_{K1}\\
&\ \ \ \ \ \ \ \ \ \ \ \ \ \ \ \ \ \ \  -T_{K1}AT_{K2}-T_{K2}AT_{K1}+T_{K3}A-AT_{K3}.
\end{aligned}$$

Now we are in a position to state the main theorem of this section.

\begin{theorem}\label{xxsec1.1}
Let $\mathcal{L}$ be a $\mathcal{J}$-subspace lattice on a Banach space $X$ over the
real or complex field $\mathbb{F}$ and $ \mathrm{Alg}\mathcal{L}$ be the associated $\mathcal{J}$-subspace
lattice algebra. Suppose that $\{L_n\}_{n=0}^{\infty}: \mathrm{Alg}\mathcal{L}\rightarrow \mathrm{Alg}\mathcal{L}$ is a family of linear mappings. Then $\{L_n\}_{n=0}^{\infty}$ satisfies
$$
L_n([A, B])=\sum_{i+j=n}[L_i(A), L_j(B)]
$$ for any $A, B\in\mathrm{Alg}\mathcal{L}$ with $AB = 0$ if and
only if for each $K \in \mathcal{J}(\mathcal{L})$, there exist a family of linear mappings $\{\Delta(\mathbf{T})_{Kn} \}_{n=1}^{\infty}: \mathrm{Alg}\mathcal{L}\rightarrow \mathrm{Alg}\mathcal{L}$ and a sequence of linear functionals $\{h_{Kn}\}_{n\in \mathbb{N}}: \mathrm{Alg}\mathcal{L}\rightarrow \mathbb{F}$ satisfying $h_{Kn}([A,B]) = 0$ whenever $AB = 0$ such
that $ L_n(A)x = (\Delta(\mathbf{T})_{Kn} (A)+ h_{Kn}(A)I)x$ for all $A \in \mathrm{Alg}\mathcal{L}$ and all $x \in K$.
\end{theorem}

\begin{proof}
The ``if'' part is obvious. We will prove the ``only if''
part. The proof will be obtianed via an induction method.

\textbf{Claim 1.} $ L_1(A)x = (\Delta(\mathbf{T})_{K1} (A)+ h_{K1}(A))x$ for all $A \in \mathrm{Alg}\mathcal{L}$  and for all $x\in K$,  and the linear
mapping $h_{K1}: \mathrm{Alg}\mathcal{L}\rightarrow \mathbb{F}$ vanishes on all commutators.

It follows from Proposition \ref{xxsec2.p1} that there is a
sequence of linear mappings $\{\delta_n\}_{n=0}^\infty :\mathrm{Alg}\mathcal{L}\rightarrow \mathrm{Alg}\mathcal{L}$ satisfying $\delta_n([A, F])=[\delta_n(A), F]+[A, \delta_n(F)]$ for all $A\in \mathrm{Alg}\mathcal{L}$, $F\in \mathcal{F}(\mathcal{L})$ such that
$$
(n+1)L_{n+1}=\sum_{k=0}^nL_{n-k} \delta_{k+1}
$$
for each non-negative integer $n$.

Clearly, the restriction of each $\delta_i$ to $\mathcal{F}(\mathcal{L})$ is a Lie derivation. This implies that it is standard by \cite[Theorem 3.1]{Lu}. That is, there exists a derivation $d_i:  \mathcal{F}(\mathcal{L})\rightarrow \mathrm{Alg}\mathcal{L}$ and a linear mapping $\tau _i: \mathcal{F}(\mathcal{L})\rightarrow Z(\mathrm{Alg}\mathcal{L})$ vanishing on every commutator such
that
\begin{equation}\label{xxsec2e10}\delta_i(F) = d_i(F) + \tau_i(F)\end{equation}
for every $F\in \mathcal{F}(\mathcal{L})$.

Let us fix an element $K\in \mathcal{J}(\mathcal{L})$ and choose $x_K \in K$. In view of Lemma \ref{xxsec2.l2}, one can take $f_K \in K_-^{\bot} $
such that
$f_K(x_K) =1$. Note that $x \otimes f_K\in \mathcal{F}(\mathcal{L})$ for all $x \in K$. Thus we can define a linear mapping $R_{Ki}: K\rightarrow K$ by
$$
R_{Ki}(x) = d_i(x \otimes f_K)x_K
$$
for all $x \in K$. Then for any $F\in \mathcal{F}(\mathcal{L})$ and any $x \in K$, we have
$$R_{Ki}(Fx) = d_i(Fx \otimes f_K)x_K = d_i(F)(x \otimes f_K)x_K + Fd_i(x \otimes f_K)x_K,$$
from which we can get
\begin{equation}\label{xxsec2et19}d_i(F)x = (R_{Ki}F - FR_{Ki})x \end{equation}for all $x \in K$. Since each
$\delta_i(i\in \mathbb{N})$ in the sequence
$\{\delta_n\}_{n=0}^\infty$ is of the form (\ref{xxsec2e10}), each $\delta_i(i\in \mathbb{N})$ can be written as
\begin{equation}\label{xxsec2et20}\delta_i(F)x =(R_{Ki}F-FR_{Ki})x + \tau_i(F)x\end{equation}
for every $F\in \mathcal{F}(\mathcal{L})$ and for all $x \in K$. Here we omit the verification of the boundedness of each $R_{Ki}$, which is similar to the proof of \cite[The Main Theorem]{Lu}.

In view of equality (\ref{xxsec2et20}) we assert
$$
\delta_i([A, F])x= (R_{Ki}(AF - FA) - (AF - FA)R_{Ki})x + \tau_i([A, F])x
$$
for all $x \in K$. On the other hand, by Lemma \ref{xxsec2.l5} we conclude
$$
\begin{aligned}
  \delta_i([A, F])x&=[\delta_i(A), F]x + [A, d_i(F)]x\\
 &=(\delta_i(A)F - F\delta_i(A))x + (A(R_{Ki}F - FR_{Ki}) - (R_{Ki}F - FR_{Ki})A)x.
\end{aligned}
$$
for all $x \in K$. Comparing the last two equalities gives
$$(\delta_i(A) -(R_{Ki}A - AR_{Ki}))Fx = F(\delta_i(A) - (R_{Ki}A - AR_{Ki}))x + \tau_i ([A, F])x.$$for all $x \in K$.
Let $y \in K$. Choosing $f \in\operatorname{dim}K_-^{\bot} $
with $f(x) = 1$ and then putting  $F =y\otimes f$ in the last equation, we
arrive at
\begin{equation}\label{xxsec2e15}(\delta_i(A)- (R_{Ki}A- AR_{Ki}))y = f((\delta_i(A)-(R_{Ki}A - AR_{Ki}))x)y +\tau_i ([A,y \otimes f])x\end{equation}
for all $y \in K$.
Note that $\tau_i ([A, y \otimes f])$ is in the center
of $ \mathrm{Alg}\mathcal{L}$, So the restriction of $\tau_i ([A, x \otimes f])$ to arbitrary $x\in K$ is in fact a scalar multiple of $x$. Let us now take $x=y$. It follows from equality (\ref{xxsec2e15}) that $\delta_i(A)- (R_{Ki}A- AR_{Ki}))y$ is a scalar multiple of $y$. Consequently, there exists a scalar $h'_{Ki}(A)$ such that
$$(\delta_i(A) - (R_{Ki}A - AR_{Ki}))y = h'_{Ki}(A)I_Ky$$
for all $y \in K$ and all $A \in \mathrm{Alg}\mathcal{L}$. Using the linearity of $\delta_i$ one can easily check that $h'_{Ki}$ is a linear mapping.
 Therefore,
\begin{equation}\label{xxsec2e20c1}\begin{aligned}L_1(A)x&= \delta_1(A)x =(R_{K1}A-AR_{K1})x +  h'_{K1}(A)I_Kx\\
&=\Delta(\mathbf{T})_{K1} (A)x+ h_{K1}(A)I_Kx\end{aligned}\end{equation}
for all $A \in \mathrm{Alg}\mathcal{L}$  and for all $x\in K$. Moreover, by the assumption on $L_1$ we know that $h_{K1}([A,B]) = 0$ whenever $AB=0$ for any $A,B \in\mathrm{Alg}\mathcal{L}$.

We now suppose that $L_k$ is a well-established mapping for each $k\leq n$. Then we need only to prove the following Claim 2.

\textbf{Claim 2.}  $ L_{n+1}(A)x = (\Delta(\mathbf{T})_{Kn+1} (A)+ h_{Kn+1}(A))x$ for all $A \in \mathrm{Alg}\mathcal{L}$, $x\in K$,
and $h_{Kn+1}: \mathrm{Alg}\mathcal{L}\rightarrow \mathbb{F}$ vanishes on all commutators.

The proof of this claim will be realized through the following two steps.

\textbf{Step 1.} $L_{n+1}(F)x=\Delta(\mathbf{T})_{K n+1}(F)x+S_{n+1}(F)x$ for all $F \in \mathcal{F}(\mathcal{L})$  and for all $x\in K$ and there exists a scalar $\lambda_{Kn+1}(F)$ such that $S_{n+1}(F)x=\lambda_{Kn+1}(F)x$.

Note that $\Delta(\mathbf{T})_{Ki}$ and $S_i$($1\leq i\leq n$)  have been well established due to Claim 1. Hence we have
\begin{equation}\label{xxsec2e20c3}
\begin{aligned}
L_{n+1}(F)x
&=\frac{1}{n+1}(L_n\delta_1 +L_{n-1}\delta_2 +\cdots +L_1\delta_{n} +L_0\delta_{n+1} )(F)x\\
&=\frac{1}{n+1}[(\Delta(\mathbf{T})_{Kn}  + S_{n}  )(R_{K1}F - FR_{K1} + h'_{K1}(F)  )\\&\hspace{10pt}+(\Delta(\mathbf{T})_{Kn-1}  + S_{n-1} )(R_{K2}F - FR_{K2} + h'_{K2}(F) ) +\cdots\\
&\hspace{10pt} +(\Delta(\mathbf{T})_{K1}  + S_1  )(R_{Kn}F - FR_{Kn}+h'_{Kn}(F))+R_{Kn+1}F - FR_{Kn+1} + h'_{Kn+1}(F) ] x\\
&=\frac{1}{n+1}[\Delta(\mathbf{T})_{Kn} (R_{K1}F-FR_{K1})+\Delta(\mathbf{T})_{Kn-1} (R_{K2}F-FR_{K2})+\cdots\\&\hspace{10pt}+\Delta(\mathbf{T})_{K1}(R_{Kn}F-FR_{Kn})+R_{Kn+1}F - FR_{Kn+1}+S'_{n+1}]x\\
&\triangleq [\Delta(\mathbf{T})_{K n+1}(F) +S_{n+1}(F)]x\\
\end{aligned}
\end{equation}
for all $F \in \mathcal{F}(\mathcal{L})$  and  all $x\in K$.

 It is also easy to verify that $S_{n+1}(F)x=\lambda_{Kn+1}(F)x$
for all $F \in \mathcal{F}(\mathcal{L})$  and  all $x\in K$;  besides,
$ \Delta(\mathbf{T})_{Kn+1}$ in the above collections is a mapping of the form ($\bigstar)$ with order $n+1$. Now let us make a simple proof of the latter. Obviously(or see \cite{Hazewinkel}), $$(\Delta(\mathbf{T})_{K1},\cdots,\Delta(\mathbf{T})_{Kn},\Delta(\mathbf{T})_{Kn+1})$$ is a  higher
derivation of order $n+1$ on $\mathcal{F}(\mathcal{L})$. We might as well assume that $\Delta'(\mathbf{T})_{Kn+1}$ is a a mapping of the form ($\bigstar)$ with order $n+1$, that is,
$$\Delta'(\mathbf{T})_{Kn+1}=([{T_K}_1,1]*[{T_K}_2, 2]*\cdots*[{T_K}_n, n]*[{T_K}'_{n+1}, n+1])_{n+1},$$
then $$(\Delta(\mathbf{T})_{K1},\cdots,\Delta(\mathbf{T})_{Kn},\Delta'(\mathbf{T})_{Kn+1})$$ is also a  higher
derivation of order $n+1$ on $\mathcal{F}(\mathcal{L})$.  A direct calculation shows that $\Delta(\mathbf{T})_{Kn+1}-\Delta'(\mathbf{T})_{Kn+1}$ is a usual derivation  on $\mathcal{F}(\mathcal{L})$(or see \cite[Lemma 4.1]{Nowicki}). Note that equality (\ref{xxsec2et19}) implies that there exists some ${T_K}''_{n+1} \in \mathcal{B}(K)$ such that
$$\Delta(\mathbf{T})_{Kn+1}-\Delta'(\mathbf{T})_{Kn+1}=[{T_K}''_{n+1}, n+1]_{n+1}.$$
Hence, setting ${T_K}'_{n+1}+{T_K}''_{n+1}={T_K}_{n+1}$, we get
$$
\begin{aligned}
\Delta(\mathbf{T})_{Kn+1}&=([{T_K}_1,1]*[{T_K}_2, 2]*\cdots*[{T_K}_n, n]*[{T_K}'_{n+1},n+1])_{n+1}+[{T_K}''_{n+1},n+1]_{n+1}\\
&=([{T_K}_1,1]*[{T_K}_2, 2]*\cdots*[{T_K}_n, n]*[{T_K}'_{n+1}+{T_K}''_{n+1},n+1])_{n+1}\\
&=([{T_K}_1,1]*[{T_K}_2, 2]*\cdots*[{T_K}_n, n]*[{T_K}_{n+1},n+1])_{n+1},
\end{aligned}
$$
which is the desired form.

\textbf{Step 2.} $L_{n+1}$ has the desired form and the Claim 2 holds.

Take any operator $A\in \mathrm{Alg}\mathcal{L} $. For any $K \in \mathcal{J}(\mathcal{L})$ and any $x \in K$, by Lemma \ref{xxsec2.l2}., there exists some $f\in K_-^{\bot} $ such that $f(x)=1$. Note that $\mathcal{F}_{\mathcal{L}}(K)$ is a ideal of $\mathrm{Alg}\mathcal{L}$, so by equality (\ref{xxsec2e20c3}) we have
$$
L_{n+1}([A,x\otimes f])x= \Delta(\mathbf{T})_{K n+1}([A,x\otimes f])x+S_{n+1}([A,x\otimes f])x\\
$$
for all $x\in K$. On the other hand, using Lemma \ref{xxsec2.l5} and the induction hypothesis, we get
$$
\begin{aligned}
 L_{n+1}([A,x\otimes f])x&=[L_{n+1}(A),x\otimes f]x+[A,L_{n+1}(x\otimes f])x\\
 &\ \ +\sum_{\substack{i+j=n+1 \\ 0<i,j<n+1}}[L_i(A), L_j(x\otimes f)]x\\
 &=[L_{n+1}(A),x\otimes f]x+[A,\Delta(\mathbf{T})_{K n+1}(x\otimes f])x\\
 &\ \ +\sum_{\substack{i+j=n+1 \\ 0<i,j<n+1}}[\Delta(\mathbf{T})_{K i}(A), \Delta(\mathbf{T})_{K j}(x\otimes f)] x
\end{aligned}
$$for all $x \in K$. From the last two relations we obtain
$$\begin{aligned}
& (L_{n+1}(A)-\Delta(\mathbf{T})_{K n+1}(A))x\\
&=f((L_{n+1}(A)-\Delta(\mathbf{T})_{K n+1}(A))x)x+S_{n+1}([A,x\otimes f])x
\end{aligned}$$
for all $x \in K$. It follows from Step 1 that $S_{n+1}([A,x\otimes f])x$ is a scalar multiple $x$.
Consequently, $(L_{n+1}(A)-\Delta(\mathbf{T})_{K n+1}(A))x$ is also a scalar multiple of $x$.
Namely,
there exists a scalar $h_{Kn+1}(A)$ such that
$$L_{n+1}(A) x= (\Delta(\mathbf{T})_{K n+1}(A)+ h_{Kn+1}(A)I_K)x$$ holds for all $A\in \mathrm{Alg}\mathcal{L} $ and all $x\in K$. Moreover, it is also easy
 to check that $h_{Kn+1} $ is linear and $h_{Kn+1}([A,B]) = 0$ for any $A,B \in\mathrm{Alg}\mathcal{L}$ with $AB=0$.
\end{proof}

Using the above theorem an immediate corollary is the following theorem form \cite{Qi}.
\begin{corollary}\label{xxsec2.c1}
Let $\mathcal{L}$ be a $\mathcal{J}$-subspace lattice on a Banach space $X$ over the
real or complex field $\mathbb{F}$ and $ \mathrm{Alg}\mathcal{L}$ be the associated $\mathcal{J}$-subspace
lattice algebra. Then a linear mapping $\delta: \mathrm{Alg}\mathcal{L}\rightarrow \mathrm{Alg}\mathcal{L}$ satisfies
$$
\delta([A, B])=[\delta(A), B]+[A, \delta(B)]
$$ for any $A, B\in\mathrm{Alg}\mathcal{L}$ with $AB = 0$ if and
only if, for each $K \in \mathcal{J}(\mathcal{L})$, there exist an operator $T_K\in \mathcal{B}(K)$ and a  linear functional $h_K: \mathrm{Alg}\mathcal{L}\rightarrow \mathbb{F}$  satisfying $h_{K}([A,B]) = 0$ whenever $AB = 0$ such
that $ \delta(A)x = (T_{K}A-AT_{K}+ h_K(A)I)x$ for all $A \in \mathrm{Alg}\mathcal{L}$  and all $x \in K$.

\end{corollary}

\section{Characterizations of $\xi$-Lie higher derivations by acting on zero
products}

Let $\mathcal{A}$ be an associative algebra over a field $\mathbb{F}$, a binary operation $[A,B]_{\xi}$ = $AB - \xi BA$ is called the \textit{$\xi$-Lie
product} of $A, B\in\mathcal{A}$ (see \cite{QiHou}). Recall that a linear mapping
$L: \mathcal{A}\rightarrow \mathcal{A}$ is called a \textit{$\xi$-Lie derivation} if $L([A,B]_{\xi}) = [L(A),B]_{\xi} + [A,L(B)]_{\xi}$
for all $A, B\in\mathcal{A} $.
In this section, we will give the characterization of $\xi$-Lie higher derivations with $\xi\neq 1$
on $\mathcal{J}$-subspace
lattice algebras by acting on zero products.

The following lemmas will be used in the sequel.

\begin{lemma}\label{xxsec3.l1}
Let $\mathcal{L}$ be a $\mathcal{J}$-subspace lattice on a Banach space $X$ over the
real or complex field $\mathbb{F}$ and $ \mathrm{Alg}\mathcal{L}$ be the associated $\mathcal{J}$-subspace
lattice algebra. Suppose that $\{L_n\}_{n=0}^{\infty}: \mathrm{Alg}\mathcal{L}\rightarrow \mathrm{Alg}\mathcal{L}$ is a family of linear mappings satisfying
$$
L_n([A, B]_{\xi})=\sum_{i+j=n}[L_i(A), L_j(B)]_{\xi}
$$ for any $A, B\in\mathrm{Alg}\mathcal{L}$ with $AB = 0$.  Then for each $n\in \mathbb{N}$, $L_n(0)=0$.
\end{lemma}

\begin{proof} In the case of $k=1$, it is easy to check that $L_1(0) = [L_1(0),0]_{\xi} + [0,L_1(0)]_{\xi}=0$. Let $s\in \mathbb{N}$ with $s\geq 1$. Assume that the lemma is true for
all $s<k$. Then by the induction hypothesis we assert
$$
L_k(0)=L_k([0, 0]_{\xi})=\sum_{\substack{i+j=k \\ 0<i,j<k}}[L_i(0), L_j(0)]_{\xi}=0.
$$
\end{proof}

\begin{lemma}\cite[Theorem 3.1]{Qi}\label{xxsec3.l2}
Let $L$ be a $\mathcal{J}$-subspace lattice on a Banach space $X$ over the real or complex field $\mathbb{F}$. Suppose that $L : \mathrm{Alg}\mathcal{L}\rightarrow \mathrm{Alg}\mathcal{L}$ is a linear mapping and $1\neq \xi \in \mathbb{F}$. Then $L$ satisfies $L([A,B]_{\xi}) = [L(A),B]_{\xi} + [A,L(B)]_{\xi}$ whenever
$ A, B \in  \mathrm{Alg}\mathcal{L}$ with $AB = 0$ if and only if one of the following statements holds.
\begin{enumerate}
\item  $\xi\neq 0$, L is a derivation of $\mathrm{Alg}\mathcal{L}$.

\item $\xi= 0, $ $L(I) \in Z(\mathrm{Alg}\mathcal{L})$ and there exists a linear derivation $\delta : \mathrm{Alg}\mathcal{L}\rightarrow \mathrm{Alg}\mathcal{L}$ such that $L(A) = \delta(A) + L(I)A$ for all $ A \in  \mathrm{Alg}\mathcal{L}$. That is, $L$
is a generalized derivation of $\mathrm{Alg}\mathcal{L}$ with associated derivation $\delta$.
\end{enumerate}
\end{lemma}

The following proposition will permit us to transfer the problem of $\xi$-Lie higher derivation with $\xi\neq 1$ by acting on zero products on $\mathcal{J}$-subspace lattice algebras into the same problems related to the corresponding $\xi$-Lie derivation.

\begin{proposition}\label{xxsec3.p1}
Let $\mathcal{L}$ be a $\mathcal{J}$-subspace lattice on a Banach space $X$ over the
real or complex field $\mathbb{F}$ and $ \mathrm{Alg}\mathcal{L}$ be the associated $\mathcal{J}$-subspace
lattice algebra. Suppose that $\{L_n\}_{n=0}^{\infty}: \mathrm{Alg}\mathcal{L}\rightarrow \mathrm{Alg}\mathcal{L}$ is a family of linear mappings satisfying
$$
L_n([A, B]_{\xi})=\sum_{i+j=n}[L_i(A), L_j(B)]_{\xi}
$$ whenever $A, B\in\mathrm{Alg}\mathcal{L}$ with $AB=0$.  Then there is a family of linear mappings
$\{\delta_n\}_{n=0}^\infty$ satisfying $\delta_n([A, B]_{\xi})=[\delta_n(A), B]_{\xi}+[A, \delta_n(B)]_{\xi}$ for any $A, B \in \mathrm{Alg}\mathcal{L}$ with $AB= 0$ such that
$$
(n+1)L_{n+1}=\sum_{k=0}^n\delta_{k+1}L_{n-k}
$$
for each non-negative integer $n$.
\end{proposition}

\begin{proof}
Let us prove it by induction on the index $n$. If $n=0$, then
$$
L_1([A, B]_{\xi})=[L_1(A), L_0(B)]_{\xi}+[L_0(A),L_1(B)]_{\xi}=[L_1(A), B]_{\xi}+[A,
L_1(B)]_{\xi}
$$
for all $A, B\in \mathrm{Alg}\mathcal{L}$ with $AB= 0$. Let us set $\delta_1=L_1$. Then
$\delta_1([A, B]_{\xi})=[\delta_1(A), B]_{\xi}+[A, \delta_1(B)]_{\xi}$ for any $A, B \in \mathrm{Alg}\mathcal{L}$ with $AB= 0$.

We now suppose that $\delta_k$ is a well-established mapping of $ \mathrm{Alg}\mathcal{L}$ for each $k\leq n$. Define
$$
\delta_{n+1}=(n+1)L_{n+1}-\sum_{k=0}^{n-1}\delta_{k+1}L_{n-k}.
$$
It is sufficient to show that $\delta_{n+1}: \mathrm{Alg}\mathcal{L}\rightarrow \mathrm{Alg}\mathcal{L}$ satisfying $$\delta_{n+1}([A, B]_{\xi})=[\delta_{n+1}(A), B]_{\xi}+[A, \delta_{n+1}(B)]_{\xi}$$ for any $A, B \in \mathrm{Alg}\mathcal{L}$ with $AB= 0$.

\textbf{Case 1.} $\xi\neq 0$.

For arbitrary elements $A, B \in \mathrm{Alg}\mathcal{L}$ with $AB= 0$, we have
$$
\begin{aligned}
\delta_{n+1}([A,B]_{\xi})&=(n+1)L_{n+1}([A,B]_{\xi})-\sum_{k=0}^{n-1}\delta_{k+1}L_{n-k}([A,B]_{\xi})\\
&=(n+1)\sum_{k=0}^{n+1}[L_k(A),
L_{n+1-k}(B)]_{\xi}\\
& \hspace{10pt}-\sum_{k=0}^{n-1}\delta_{k+1}\left(\sum_{i=0}^{n-k}[L_i(A),
L_{n-k-i}(B)]_{\xi}\right).
\end{aligned}
$$

Using the induction hypothesis and Lemma \ref{xxsec3.l2}, we obtain
$$
\begin{aligned}
\delta_{n+1}([A, B]_{\xi})&=\sum_{k=0}^{n+1}(k+n+1-k)[L_k(A),
L_{n+1-k}(B)]_{\xi}-\\
&\hspace{10pt}\sum_{k=0}^{n-1}\delta_{k+1}\left(\sum_{i=0}^{n-k}[L_i(A), L_{n-k-i}(B)]_{\xi}\right)\\
&=\sum_{k=0}^{n+1}k[L_k(A),
L_{n+1-k}(B)]_{\xi}-\sum_{k=0}^{n-1}\sum_{i=0}^{n-k}[\delta_{k+1}(L_i(A)),
L_{n-k-i}(B)]_{\xi}\\
&\hspace{10pt}+\sum_{k=0}^{n+1}(n+1-k)[L_k(A),
L_{n+1-k}(B)]_{\xi}-\sum_{k=0}^{n-1}\sum_{i=0}^{n-k}[L_i(A),
\delta_{k+1}(L_{n-k-i}(B))]_{\xi} .
\end{aligned}
$$
Let us write
$$
\begin{aligned}
U&=\sum_{k=0}^{n+1}k[L_k(A), L_{n+1-k}(B)]_{\xi}-\sum_{k=0}^{n-1}\sum_{i=0}^{n-k}[\delta_{k+1}(L_i(A)), L_{n-k-i}(B)]_{\xi},\\
V&=\sum_{k=0}^{n+1}(n+1-k)[L_k(A),
L_{n+1-k}(B)]_{\xi}-\sum_{k=0}^{n-1}\sum_{i=0}^{n-k}[L_i(A),
\delta_{k+1}(L_{n-k-i}(B))]_{\xi}.
\end{aligned}
$$
Then $\delta_{n+1}([A,B])=U+V$. In the expression of sum $\sum_{k=0}^{n-1}\sum_{i=0}^{n-k}$, we notice that $k\neq n$ and
$0\leq k+i\leq n$. If we set $r=k+i$, then
$$\begin{aligned}
U&=\sum_{k=0}^{n+1}k[L_k(A), L_{n+1-k}(B)]_{\xi}-\sum_{r=0}^{n}\sum_{0\leq k\leq r,k\neq n}[\delta_{k+1}(L_{r-k}(A)),L_{n-r}(B)]_{\xi}\\
&=\sum_{r=0}^{n}(r+1)[L_{r+1}(A),
L_{n-r}(B)]_{\xi}-\sum_{r=0}^{n-1}\sum_{k=0}^{r}[\delta_{k+1}(L_{r-k}(A)),
L_{n-r}(B)]_{\xi}\\
&\hspace{10pt}- \sum_{k=0}^{n-1}[\delta_{k+1}(L_{n-k}(A)), B]_{\xi}\\
&=\sum_{r=0}^{n-1}[(r+1)L_{r+1}(A)-\sum_{k=0}^{r}\delta_{k+1}(L_{r-k}(A)),L_{n-r}(B)]_{\xi}\hspace{45pt}\\
&\quad +(n+1)[L_{n+1}(A),
B]_{\xi}-\sum_{k=0}^{n-1}[\delta_{k+1}(L_{n-k}(A)), B]_{\xi}.
\end{aligned}$$
By the induction hypothesis $
(r+1)L_{r+1}(A)=\sum_{k=0}^{r}\delta_{k+1}(L_{r-k}(A))
$
for $r=0,\cdots, n-1$, we get
$$
U=[(n+1)L_{n+1}(A)-\sum_{k=0}^{n-1}\delta_{k+1}(L_{n-k}(A)),
B]_{\xi}=[\delta_{n+1}(A), B]_{\xi}.
$$
Similarly, one can deduce that $V=[A, \delta_{n+1}(B)]_{\xi}$. We therefore conclude
$$
\delta_{n+1}([A, B]_{\xi})=U+V=[\delta_{n+1}(A), B]_{\xi}+[A, \delta_{n+1}(B)]_{\xi}
$$
for any $A, B \in \mathrm{Alg}\mathcal{L}$ with $AB=0$, which is the desired result.

\textbf{Case 2.} $\xi=0$.

Note that the sequence $\{\delta_{k+1}\}_{k=0}^{n-1}$ is a family of generalized derivations by Lemma \ref{xxsec3.l2}. That is, for an arbitrary $k=0, 1,\cdots, n-1$, $\delta_{k+1}(A)=\tau_{k+1}(A)+\delta_{k+1}(I)(A)$ for all $A \in \mathrm{Alg}\mathcal{L}$.
Hence we have
$$
\begin{aligned}
\delta_{n+1}([A,B]_{\xi})&=(n+1)L_{n+1}([A,B]_{\xi})-\sum_{k=0}^{n-1}\delta_{k+1}L_{n-k}([A,B]_{\xi})\\
&=(n+1)\sum_{k=0}^{n+1}[L_k(A),
L_{n+1-k}(B)]_{\xi}-\sum_{k=0}^{n-1}\delta_{k+1}\left(\sum_{i=0}^{n-k}[L_i(A),
L_{n-k-i}(B)]_{\xi}\right)
\end{aligned}
$$
for any $A, B \in \mathrm{Alg}\mathcal{L}$ with $AB=0$. In view of Lemma \ref{xxsec3.l1} we know that $L_{n-k}([A,B]_{\xi})=0$ in the above relation. Using the induction hypothesis one can compute
$$
\begin{aligned}
\delta_{n+1}([A, B]_{\xi})&=\sum_{k=0}^{n+1}(k+n+1-k)[L_k(A),
L_{n+1-k}(B)]_{\xi}\\
&\hspace{10pt}-\sum_{k=0}^{n-1}\tau_{k+1}\left(\sum_{i=0}^{n-k}[L_i(A), L_{n-k-i}(B)]_{\xi}\right)\\
&=\sum_{k=0}^{n+1}k[L_k(A),
L_{n+1-k}(B)]_{\xi}-\sum_{k=0}^{n-1}\sum_{i=0}^{n-k}[\tau_{k+1}(L_i(A)),
L_{n-k-i}(B)]_{\xi}\\
&\hspace{10pt}+\sum_{k=0}^{n+1}(n+1-k)[L_k(A),
L_{n+1-k}(B)]_{\xi}-\sum_{k=0}^{n-1}\sum_{i=0}^{n-k}[L_i(A),
\tau_{k+1}(L_{n-k-i}(B))]_{\xi}\\
&=\sum_{k=0}^{n+1}k[L_k(A),
L_{n+1-k}(B)]_{\xi}-\sum_{k=0}^{n-1}\sum_{i=0}^{n-k}[\delta_{k+1}(L_i(A)),
L_{n-k-i}(B)]_{\xi}\\
&\hspace{10pt}+\sum_{k=0}^{n+1}(n+1-k)[L_k(A),
L_{n+1-k}(B)]_{\xi}-\sum_{k=0}^{n-1}\sum_{i=0}^{n-k}[L_i(A),
\delta_{k+1}(L_{n-k-i}(B))]_{\xi}
\end{aligned}
$$
for any $A, B \in \mathrm{Alg}\mathcal{L}$ with $AB= 0$. If we set
$$
U=\sum_{k=0}^{n+1}k[L_k(A), L_{n+1-k}(B)]_{\xi}-\sum_{k=0}^{n-1}\sum_{i=0}^{n-k}[\delta_{k+1}(L_i(A)), L_{n-k-i}(B)]_{\xi}$$
and
$$
V=\sum_{k=0}^{n+1}(n+1-k)[L_k(A),
L_{n+1-k}(B)]_{\xi}-\sum_{k=0}^{n-1}\sum_{i=0}^{n-k}[L_i(A),
\delta_{k+1}(L_{n-k-i}(B))]_{\xi}.
$$
Then $\delta_{n+1}([A,B])=U+V$ for any $A, B \in \mathrm{Alg}\mathcal{L}$ with $AB= 0$. Similarly, we can show that
$U=[\delta_{n+1}(A), B]_{\xi}$ and $V=[A, \delta_{n+1}(B)]_{\xi}$. Thus we have
$$
\delta_{n+1}([A, B]_{\xi})=P+Q=[\delta_{n+1}(A), B]_{\xi}+[A, \delta_{n+1}(B)]_{\xi}
$$
for any $A, B \in \mathrm{Alg}\mathcal{L}$ with $AB= 0$.
\end{proof}

Now we are in a position to state the main theorem of this section.

\begin{theorem}\label{xxsec3t1}
Let $\mathcal{L}$ be a $\mathcal{J}$-subspace lattice on a Banach space $X$ over the
real or complex field $\mathbb{F}$  and $ \mathrm{Alg}\mathcal{L}$ be the associated $\mathcal{J}$-subspace
lattice algebra. Let $\xi\in \mathbb{F}$ with $\xi\neq1$. Suppose that $\mathcal{G}$ is the
set of all $\xi$-Lie higher derivations on $ \mathrm{Alg}\mathcal{L}$ by acting on zero products and $\mathcal{H}$
be the set of all sequences of $\xi$-Lie derivations on $ \mathrm{Alg}\mathcal{L}$ by acting on zero products
with first component zero. Then there is a one-to-one correspondence
between $\mathcal{G}$ and $\mathcal{H}$.
\end{theorem}

\begin{proof}
It follows from Proposition \ref{xxsec3.p1} that for an arbitrary
 $G=\{L_n\}_{n=0}^{\infty}\in\mathcal{G}$ on $ \mathrm{Alg}\mathcal{L}$ satisfying $$
L_n([A, B]_{\xi})=\sum_{i+j=n}[L_i(A), L_j(B)]_{\xi}
$$ whenever $A, B\in\mathrm{Alg}\mathcal{L}$ with $AB = 0$, there is a sequence
$D=\{\delta_n\}_{n=0}^\infty$ of linear maps satisfying $\delta_n([A, B]_{\xi})=[\delta_n(A), B]_{\xi}+[A, \delta_n(B)]_{\xi}$ for any $A, B \in \mathrm{Alg}\mathcal{L}$ with $AB= 0$ on $\mathrm{Alg}\mathcal{L}$ with $\delta_0=0$ such that
$$
(n+1)L_{n+1}=\sum_{k=0}^n \delta_{k+1}L_{n-k}
$$
for each non-negative integer $n$. Hence the following mapping
$$
\begin{aligned}
\varphi: \mathcal{G} &\longrightarrow \mathcal{H} \\
\{L_n\}_{n=0}^{\infty}=G &\longmapsto
D=\{\delta_n\}_{n=0}^{\infty}
\end{aligned}
$$
is well-defined. Note that the solution of the recursive
relation of Proposition \ref{xxsec3.p1} is unique. Therefore
$\varphi$ is injective.

We next prove that $\varphi$ is also surjective. For a given
sequence $D=\{\delta_n\}_{n=0}^\infty$ of Lie derivations with
$\delta_0=0$, one can define $L_0=I$ and
$$
(n+1)L_{n+1}=\sum_{k=0}^n \delta_{k+1}L_{n-k}
$$
for each $n$. It is sufficient to show
that $G=\{L_n\}_{n=0}^{\infty}$ is a $\xi$-Lie higher derivation on $ \mathrm{Alg}\mathcal{L}$ by acting on zero products.

Obviously, $L_1=\delta_1$ is a $\xi$-Lie derivation on $ \mathrm{Alg}\mathcal{L}$ by acting on zero products. Assume that $L_k([A,B]_{\xi})=\sum_{i=0}^k[L_i(A),
L_{k-i}(B)]_{\xi}$ for any $A,B\in \mathrm{Alg}\mathcal{L}$ with $AB=0$ and for each $k\leq n$.
Note that
$$\begin{aligned}
(n+1)L_{n+1}([A,  B]_{\xi})&=\sum_{k=0}^n \delta_{k+1}L_{n-k}([A,  B]_{\xi})\\
&=\sum_{k=0}^n \delta_{k+1}\left(\sum_{i=0}^{n-k}[L_i(A),
L_{n-k-i}( B)]_{\xi}\right)
\end{aligned}$$
for any $A,B\in \mathrm{Alg}\mathcal{L}$ with $AB=0$.

\textbf{Case 1.}  $\xi\neq 0$

Using Lemma \ref{xxsec3.l2} and the induction hypothesis, we compute that
$$
\begin{aligned}
(n+1)L_{n+1}([A, B]_{\xi})&=\sum_{k=0}^n \sum_{i=0}^{n-k}\left\{[\delta_{k+1}(L_i(A)), L_{n-k-i}(B)]_{\xi}+[L_i(A), \delta_{k+1}(L_{n-k-i}(B))]_{\xi}\right\}\\
&=\sum_{i=0}^n[\sum_{k=0}^{n-i}\delta_{k+1}L_{n-i-k}(A), L_i(B)]_{\xi}+
\sum_{i=0}^n[L_i(A), \sum_{k=0}^{n-i}\delta_{k+1}L_{n-i-k}(B)]_{\xi}\\
&=\sum_{i=0}^n[(n-i+1)L_{n-i+1}(A), L_i(B)]_{\xi}+\sum_{i=0}^n[L_i(A), (n-i+1)L_{n-i+1}(B)]_{\xi}\\
&=\sum_{i=1}^{n+1}i[L_i(A), L_{n+1-i}(B)]_{\xi}+\sum_{i=0}^n(n+1-i)[L_i(A), L_{n-i+1}(B)]_{\xi}\\
&=(n+1)\sum_{k=0}^{n+1}[L_k(A), L_{n+1-k}(B)]_{\xi}
\end{aligned}
$$
for any $A,B\in \mathrm{Alg}\mathcal{L}$ with $AB=0$.

\textbf{Case 2.}  $\xi=0$

By Lemma \ref{xxsec3.l2} it follows that $\delta_n$ is a generalized derivation with associated derivation $\tau_n$ for each $n\in \mathbb{N}$. That is, $\delta_n(A)=\tau_n(A)+\delta_n(I)A$ holds for all $A\in \mathrm{Alg}\mathcal{L}$. In an analogous manner, one can show
$$
\begin{aligned}
(n+1)L_{n+1}([A, B]_{\xi})&=\sum_{k=0}^n \sum_{i=0}^{n-k}\left\{[\tau_{k+1}(L_i(A)), L_{n-k-i}(B)]_{\xi}+[L_i(A), \tau_{k+1}(L_{n-k-i}(B))]_{\xi}\right\}\\
&=\sum_{i=0}^n[\sum_{k=0}^{n-i}\tau_{k+1}L_{n-i-k}(A), L_i(B)]_{\xi}+
\sum_{i=0}^n[L_i(A), \sum_{k=0}^{n-i}\tau_{k+1}L_{n-i-k}(B)]_{\xi}\\
&=\sum_{i=0}^n[\sum_{k=0}^{n-i}\delta_{k+1}L_{n-i-k}(A), L_i(B)]_{\xi}+
\sum_{i=0}^n[L_i(A), \sum_{k=0}^{n-i}\delta_{k+1}L_{n-i-k}(B)]_{\xi}\\
&=\sum_{i=0}^n[(n-i+1)L_{n-i+1}(A), L_i(B)]_{\xi}+\sum_{i=0}^n[L_i(A), (n-i+1)L_{n-i+1}(B)]_{\xi}\\
&=\sum_{i=1}^{n+1}i[L_i(A), L_{n+1-i}(B)]_{\xi}+\sum_{i=0}^n(n+1-i)[L_i(A), L_{n-i+1}(B)]_{\xi}\\
&=(n+1)\sum_{k=0}^{n+1}[L_k(A), L_{n+1-k}(B)]_{\xi}
\end{aligned}
$$
for any $A,B\in \mathrm{Alg}\mathcal{L}$ with $AB=0$.

In any case, one can get
$$
\begin{aligned}
L_{n+1}([A, B]_{\xi})=\sum_{k=0}^{n+1}[L_k(A), L_{n+1-k}(B)]_{\xi}
\end{aligned}
$$
for any $A,B\in \mathrm{Alg}\mathcal{L}$ with $AB=0$. This shows that
$G=\{L_n\}_{n=0}^{\infty}$ is a $\xi$-Lie higher derivation of
$\mathrm{Alg}\mathcal{L}$. Thus $G\in \mathcal{G}$ and this completes the proof.
\end{proof}

Before stating the second main result in this section, we need a conclusion that characterizes generalized higher derivations in terms of generalized derivations. 

\begin{lemma}\label{xxsec3.l5}
Let $\mathcal{A}$ be an associative algebra and $G=\{L_n\}_{n=0}^{\infty}$ be a generalized higher derivation with an associated higher derivation
$D=\{d_n\}_{n=0}^{\infty}$.
Then there is a family of generalized derivations $\{\gamma_n\}_{n\in \mathbb{N}}$ with the family of associated derivations $\{\tau_n\}_{n\in \mathbb{N}}$ such that
$$
(n+1)L_{n+1}=\sum_{k=0}^n \gamma_{k+1}L_{n-k},\ \
(n+1)d_{n+1}=\sum_{k=0}^n \tau_{k+1}d_{n-k}$$
for each nonnegative integer $n$.
\end{lemma}

\begin{proof}
Let us show it by induction on $n$. If $n=0$, then
$$
L_1(xy)=L_1(x)y+xd_1(y)
$$
for all $x, y\in\mathcal{A}$. Let us write $\gamma_1=L_1$ and $\tau_1=d_1$.
Then $\gamma_1$ is a generalized derivation of $\mathcal{A}$ with associated derivation $\tau_1$.

Suppose that $\gamma_k$ is a generalized derivation of $\mathcal{A}$ with associated derivation $\tau_k$ and that associated derivation $\tau_k$ satisfies
$kd_{k}=\sum_{s=0}^{k-1}\tau_{s+1}d_{k-1-s}$
for each $k\leq n$. Let us define
$$
\gamma_{n+1}=(n+1)L_{n+1}-\sum_{k=0}^{n-1}\gamma_{k+1}L_{n-k}.
$$
It is sufficient to prove that $\gamma_{n+1}$ is a generalized derivation of $\mathcal{A}$.

For any $x,y \in\mathcal{A}$, we have
$$\begin{aligned}
\gamma_{n+1}( x  y )&=(n+1)L_{n+1}(xy)-\sum_{k=0}^{n-1}\gamma_{k+1}L_{n-k}(xy)\\
&=(n+1)\sum_{k=0}^{n+1}L_k(x)d_{n+1-k}(y) -\sum_{k=0}^{n-1}\gamma_{k+1}\left(\sum_{i=0}^{n-k}L_i(x)d_{n-k-i}(y) \right).
\end{aligned}
$$
Therefore,
$$\begin{aligned}
\gamma_{n+1} (x  y)  &=\sum_{k=0}^{n+1}(k+n+1-k) L_k(x)  d_{n+1-k}(y) -
\sum_{k=0}^{n-1}\gamma_{k+1}\left(\sum_{i=0}^{n-k} L_i(x)  d_{n-k-i}(y) \right)\\
&=\sum_{k=0}^{n+1}k L_k(x)  d_{n+1-k}(y) +\sum_{k=0}^{n+1} (n+1-k)L_k(x)d_{n+1-k}(y) \\
&\quad -\sum_{k=0}^{n-1}\sum_{i=0}^{n-k}\{ \gamma_{k+1} (L_i(x))   d_{n-k-i}(y) + L_i(x)  \tau_{k+1} (d_{n-k-i}(y))  \}.
\end{aligned}$$
Let us write
$$\begin{aligned}
U&=\sum_{k=0}^{n+1}k L_k(x)  d_{n+1-k}(y) -\sum_{k=0}^{n-1}\sum_{i=0}^{n-k} \gamma_{k+1}(L_i(x))  d_{n-k-i}(y),  \\
V&=\sum_{k=0}^{n+1} (n+1-k)L_k(x)d_{n+1-k}(y) -\sum_{k=0}^{n-1}\sum_{i=0}^{n-k} L_i(x)  \tau_{k+1}(d_{n-k-i}(y)).
\end{aligned}$$
Thus $\gamma_{n+1}[x, y]=U+V$. In the summands $\sum_{k=0}^{n-1}\sum_{i=0}^{n-k}$, we know that $0\leq k+i\leq n$
and $k\neq n$. If we set $r=k+i$, then
$$\begin{aligned}
U&=\sum_{k=0}^{n+1}k L_k(x)  d_{n+1-k}(y) -\sum_{r=0}^{n}\sum_{0\leq k\leq r,  k\neq n} \gamma_{k+1}(L_{r-k}(x))  d_{n-r}(y) \\
&=\sum_{k=0}^{n+1}k L_k(x)  d_{n+1-k}(y) -\sum_{r=0}^{n-1}\sum_{k=0}^{r} \gamma_{k+1}(L_{r-k}(x))  d_{n-r}(y) -
\sum_{k=0}^{n-1} \gamma_{k+1}(L_{n-k}(x))  y \\
&=\sum_{r=0}^{n}(r+1) L_{r+1}(x)  d_{n-r}(y) -\sum_{r=0}^{n-1}\sum_{k=0}^{r} \gamma_{k+1}(L_{r-k}(x)) d_{n-r}(y) -
\sum_{k=0}^{n-1} \gamma_{k+1}(L_{n-k}(x))  y \\
&=\sum_{r=0}^{n-1}[ (r+1)L_{r+1}(x)-\sum_{k=0}^{r}\gamma_{k+1}(L_{r-k}(x)) ] d_{n-r}(y) \\
&\quad +(n+1) L_{n+1}(x)  y -\sum_{k=0}^{n-1} \gamma_{k+1}(L_{n-k}(x))  y .
\end{aligned}$$
By the induction hypothesis, $(r+1)L_{r+1}(x)=\sum_{k=0}^{r}\gamma_{k+1}(L_{r-k}(x))$ for $r=0,\cdots,n-1$.
Therefore, we deduce that
$$
U= [(n+1)L_{n+1}(x)-\sum_{k=0}^{n-1}\gamma_{k+1}(L_{n-k}(x))]  y=\gamma_{n+1}(x)  y.
$$
On the other hand, a direct computation shows that
$$\begin{aligned}
V&=\sum_{k=0}^{n+1} (n+1-k)L_k(x)d_{n+1-k}(y) -\sum_{k=0}^{n-1}\sum_{i=0}^{n-k} L_i(x)  \tau_{k+1}(d_{n-k-i}(y)) \\
&=\sum_{k=0}^{n+1} (n+1-k)L_k(x)d_{n+1-k}(y) -\sum_{i=0}^{n}\sum_{k=0}^{n-i} L_i(x)  \tau_{k+1}(d_{n-k-i}(y)) + x  \tau_{n+1}(y) \\
&=\sum_{i=0}^{n} (n+1-i)L_i(x)d_{n+1-i}(y) -\sum_{i=0}^{n}\sum_{k=0}^{n-i} L_i(x)  \tau_{k+1}(d_{n-k-i}(y)) + x  \tau_{n+1}(y).
\end{aligned}$$
Using the induction hypothesis again, we obtain
$$\begin{aligned}
V&=\sum_{i=0}^{n}L_i(x)[(n+1-i)d_{n+1-i}(y)-\sum_{k=0}^{n-i}\tau_{k+1}(d_{n-k-i}(y))] + x  \tau_{n+1}(y) \\
&=x[(n+1)d_{n+1}(y)-\sum_{k=0}^{n-1}\tau_{k+1}(d_{n-k}(y))-\tau_{n+1}(y)] + x  \tau_{n+1}(y) \\
&= x  \tau_{n+1}(y) ,
\end{aligned}$$
where $\tau_{n+1}=(n+1)d_{n+1}-\sum_{k=0}^{n-1}\tau_{k+1}d_{n-k}$.
Hence $$\gamma_{n+1}(xy)=\gamma_{n+1}(x)y+x  \tau_{n+1}(y).$$ Whence $\gamma_{n+1}$ is a generalized  derivation of $\mathcal{A}$ with the associated derivation $\tau_{n+1}$ and
this completes the proof.
\end{proof}

The second main result in this section reads as follows.

\begin{theorem}\label{xxsec2.1}
Let $\mathcal{L}$ be a $\mathcal{J}$-subspace lattice on a Banach space $X$ over the
real or complex field $\mathbb{F}$ and $\mathrm{Alg}\mathcal{L}$ be the associated $\mathcal{J}$-subspace
lattice algebra. Suppose that $\{L_n\}_{n=0}^{\infty}: \mathrm{Alg}\mathcal{L}\rightarrow \mathrm{Alg}\mathcal{L}$ is a family of linear mappings and $\xi\in \Bbb{F}$ with $\xi\neq 1$. Then $\{L_n\}_{n=0}^{\infty}$ satisfies
$$
L_n([A, B]_{\xi})=\sum_{i+j=n}[L_i(A), L_j(B)]_{\xi}
$$ for any $A, B\in\mathrm{Alg}\mathcal{L}$ with $AB = 0$ if and only if one of the following statements hold.
\begin{enumerate}
\item $\xi\neq 0, $  $\{L_n\}_{n=0}^{\infty}$ is a higher derivation of $\mathrm{Alg}\mathcal{L}$.
 \item $\xi= 0, $  $\{L_n\}_{n=0}^{\infty}$ is a generalized higher derivation of $\mathrm{Alg}\mathcal{L}$.
 \end{enumerate}
\end{theorem}
\begin{proof}

\textbf{Case 1.} (a) $\xi\neq 0$

The proof can be obtained by using Lemma \ref{xxsec3.l2}, Theorem \ref{xxsec3t1} and \cite[Theorem 2.4]{Han1}.

\textbf{Case 2.} (b) $\xi= 0$

The ``if'' part can be obtained from Lemma \ref{xxsec3.l2}, Lemma \ref{xxsec3.l5} and Theorem \ref{xxsec3t1}.
We need only to  consider the ``only if'' part.

Proposition \ref{xxsec3.p1} means that there is a family $\{\delta_n\}_{n=0}^\infty$ of linear mappings satisfying $\delta_n([A, B]_{\xi})=[\delta_n(A), B]_{\xi}+[A, \delta_n(B)]_{\xi}$ whenever $A, B \in \mathrm{Alg}\mathcal{L}$ with $AB= 0$ such that
$
(n+1)L_{n+1}=\sum_{k=0}^n\delta_{k+1}L_{n-k}
$
for each nonnegative integer $n$. Now each $\delta_n$ is a generalized derivation, which is due to Lemma \ref{xxsec3.l2}. We might as well assume the associated derivation of $\delta_n$ be $\tau_n$ for each nonnegative integer $n$.

To prove $\{L_n\}_{n=0}^{\infty}$ is a generalized higher derivation, let us take an inductive approach for the index $n$. It is clear that $L_1=\delta_1$ is a generalized derivation of $\mathrm{Alg}\mathcal{L}$ with the associated derivation $\tau_1$. Suppose that
$L_k(AB)=\sum_{i=0}^kL_i(A)d_{k-i}(B)$ for all $A, B \in \mathrm{Alg}\mathcal{L}$ and
$kd_k=\sum_{s=0}^{k-1}\tau_{s+1}d_{k-1-s}$ for all $k\leq n$.

Let us now prove that $L_{n+1}(AB)=\sum_{i=0}^{n+1}L_i(A)d_{n+1-i}(B)$ for all $A, B \in \mathrm{Alg}\mathcal{L}$. Note that
$$\begin{aligned}
(n+1)L_{n+1}(AB)&=\sum_{k=0}^n \delta_{k+1}L_{n-k}(AB)\\
&=\sum_{k=0}^{n }\sum_{i=0}^{n-k}\delta_{k+1}(L_i(A)
d_{n-k-i}(B))\\&=\sum_{k=0}^n \sum_{i=0}^{n-k}\left\{[\delta_{k+1}(L_i(A))d_{n-k-i}(B)+L_i(A)\tau_{k+1}(d_{n-k-i}(B))]\right\}\\&=\sum_{i=0}^n\left(\sum_{k=0}^{n-i}\delta_{k+1}L_{n-i-k}(A)\right)d_i(B)+
\sum_{i=0}^nL_i(A)\left(\sum_{k=0}^{n-i}\tau_{k+1}d_{n-i-k}(B)\right).
\end{aligned}$$
So by  the induction hypothesis we can  do the following computation.
$$
\begin{aligned}
(n+1)L_{n+1}(AB)
&=\sum_{i=0}^n(n-i+1)L_{n-i+1}(A)d_i(B)+A \sum_{k=0}^{n}\tau_{k+1}d_{n-k}(B)\\
&\ +\sum_{i=1}^nL_i(A) (n-i+1)d_{n-i+1}(B)\\\
&=\sum_{i=1}^{n+1}iL_i(A)d_{n+1-i}(B)+A \sum_{k=0}^{n}\tau_{k+1}d_{n-k}(B)\\
&\ +\sum_{i=1}^n(n+1-i)L_i(A)d_{n-i+1}(B)\\
&=\sum_{i=1}^{n+1}(n+1)L_i(A)d_{n+1-i}(B)+A \sum_{k=0}^{n}\tau_{k+1}d_{n-k}(B)\\
&=(n+1)\sum_{k=0}^{n+1}L_k(A)d_{n+1-k}(B),
\end{aligned}
$$
where $(n+1)d_{n+1}=\sum_{k=0}^{n}\tau_{k+1}d_{n-k}$. This completes the
proof.

\end{proof}

\end{document}